\documentclass[12pt]{article}

\usepackage{fancyhdr, amsmath, amssymb, amsfonts, url, subfigure, dsfont, mathrsfs, graphicx, epsfig, subfigure, amsthm, tikz}
    \usepackage{epsfig}
    \usepackage{graphicx}
    \usepackage{color}
 
\newlength{\guillotine}
\newtheorem{thm}{Theorem}[section]
\newtheorem{cor}[thm]{Corollary}
\newtheorem{lemma}[thm]{Lemma}

\newtheorem{definition}[thm]{Definition}
\newtheorem{example}[thm]{Example}

\theoremstyle{remark}
\newtheorem{rem}[thm]{Remark}
\begin{document}

\title{Effective intrinsic ergodicity  for expanding interval maps}
\author{Mark Pollicott\footnote{Department of Mathematics, Warwick University, Coventry, CV8 2AH}}
\date{}

\maketitle


\section{ Introduction}

The variational principle is one of the central pillars of smooth ergodic theory and thermodynamic formalism.  It was originally formulated for hyperbolic systems by Ruelle \cite{ruelle} and proved in full generality by Walters \cite{walters}.  
It relates the thermodynamic  pressure function to entropies and integrals with respect to invariant measures.  

We will consider the particular case where  $T: I  \to I$ is  a piecewise $C^2$ mixing expanding map of the interval $I = [0,1)$ and 
where  $\phi: I \to  R$ is  a H\"older continuous  function.

\begin{definition}

We can denote by $m_\phi$  the unique equilibrium state associated to $\phi$, i.e., $m$ is the unique probability measure realising 
the following supremum
$$
P(\phi):= \sup \left\{
h(\mu) + \int \phi d\mu
\hbox{ : } \mu = \hbox{$T$-invariant probability}
\right\},
$$
where $h(\mu) $ is the entropy
(i.e., the variational principle).
\end{definition}


 For definiteness, we will consider the following well known examples.
 
 \medskip
 \noindent

\begin{example}
Let $\beta > 1$ and consider $T: [0,1) \to [0,1)$ defined by 
$$
T(x) = \beta x \qquad  \hbox{\rm (mod $1$)}
$$
then $T$ is called a $\beta$-transformation.  
This is piecewise affine  on the intervals 
$$\left[0, \frac{1}{\beta} \right), \left[\frac{1}{\beta}, \frac{2}{\beta}\right), 
\cdots
 \left[\frac{[1/\beta]-1}{\beta}, \frac{[1/\beta]}{\beta}\right), 
  \left[ \frac{[1/\beta]}{\beta}, 1\right)$$
  where $[ \cdot ]$ denotes the integer part of a real number.
 \end{example}

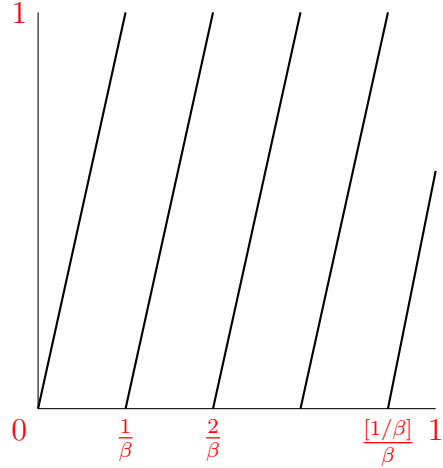
\begin{figure}
\centerline{
\begin{tikzpicture}[scale=1.05]
\draw (0,0) -- (0,5);
\draw (0,0) -- (5,0);
\draw[thick] (0,0) -- (1.1,5);
\draw[thick] (1.1,0) -- (2.2,5);
\draw[thick] (2.2,0) -- (3.3,5);
\draw[thick] (3.3,0) -- (4.4,5);
\draw[thick] (4.4,0) -- (5,3);
\node [below left, red] at (0,0) {$0$};
\node [below, red] at (5,0) {$1$};
\node [left, red] at (0,5) {$1$};
\node [below, red] at (1.1,0) {$\frac{1}{\beta}$};
\node [below, red] at (2.2,0) {$\frac{2}{\beta}$};
\node [below, red] at (4.4,0) {$\frac{[1/\beta]}{\beta}$};
\end{tikzpicture}
}
\caption{Graph of a $\beta$-transformation}
\end{figure}

\bigskip

We claim the following analogue of the 
Einseidler-Kaydev-Polo inequality
(originally established for Markov expanding maps, subshifts of finite type and Anosov diffeomorphisms \cite{kadyrov}, \cite{ruhr}) also holds in this context.

\begin{thm}\label{main}
There exists a constant $C_0 = C_0(\phi)$ such that for any $T$-invariant probability $\mu$ we have 
$$
\left| \int f d\mu - \int f dm_\phi \right|
\leq C_0 \|f\|
\sqrt{
P(\phi)
- \left(h(\mu) + \int \phi d\mu\right)} \eqno(1.1)
$$
where $\|\cdot\| $ is the norm of bounded variation.
\end{thm}

 In the special case that $\beta \in \mathbb N$
 (or more generally the orbit of $1/\beta$ is finite)  the map is Markov.
In this case Theorem \ref{main}  would be a consequence of Kaydev's theorem \cite{kadyrov}..

\begin{example}[Parry measure]
 If we take $\phi=0$ then $P(0) = h(T)$ is the topological entropy and 
 the equilibrium state is the unique measure
$m_0$ which maximizes the the entropy. 
In particular, 
Parry showed that  $m_0$ is absolutely continuous with  
density $\rho: [0,1] \to \mathbb R^+$ given by 
$$
\rho(x) = 
\frac{
\sum_{n : T^n(1) > x} 
\beta^{-n}
}{
\int \left(  \sum_{n :  T^n(1) > x} 
\beta^{-n}
\right) dx}
$$
with  normalization constant $K>0$ \cite{parry}.
\end{example}

Theorem \ref{main} now has the following corollary (when $\phi=0$).

\begin{cor}
There exists a constant $C_0>0$ such that for any $T$-invariant measure $\mu$ we have 
$$
\left| \int f d\mu - \int f dm_0 \right|
\leq C_0 \|f\|
\sqrt{h(T)  - h(\mu)}.
\eqno(1.2)
$$
\end{cor}

\medskip
\noindent
{\bf  A little history.}
A version of this result was apparently first proved in the thesis of Polo for doubling maps \cite{polo}, where it was attributed to Einseidler.  The above theorem was proved by Kadyrov for finite state  shift spaces
when $\phi=0$ (which was called effective intrinsic ergodicity) \cite{kadyrov}.
This was extended to H\"older potentials and  infinite state shift spaces by Ruhr \cite{ruhr}.  Subsequently, 
 Ruhr-Sarig gave an alternative proof
 and a local version  where the upper bound has  the variance replacing the norm
   \cite{ruhrsarig}.


\medskip

\section{Proof of  Theorem 1.2}

We begin by recalling the definition of the bounded variation semi-norm of a function $\psi: I \to \mathbb R$ which  takes the form  
$$
\|\psi\|_{BV} = \sup\left\{ \left|\sum_{i=0}^n \psi(x_i) - \psi(x_{i+1})\right|
\hbox{ : } 0 = x_0 < x_1 \cdots < x_n < x_{n+1} = 1
 \right\}
$$
and let $\|\psi\|_{L^1} = \int_I |\psi(x)| dx$ denote the $L^1$-norm.  
We let $BV(I)$ denote the Banach space 
of measurable functions $\psi: I \to \mathbb C$ with norm $\|\psi\| := \|\psi\|_{BV}
 + \|\psi\|_{L^1}  < +\infty$.

\begin{definition}
Let $T: I \to I$ be a monotone piecewise continuous map.  
Given  $\phi \in BV(I)$
we can define the {\it  transfer operator}
$\mathcal L_\phi : BV(I) \to BV(I)$ by
$$
\mathcal L_\phi w(x) = \sum_{Ty=x} e^{\phi(y)} w(y).  
$$
\end{definition}

Under  additional assumptions on the function $\phi$ the operator 
$\mathcal L_\phi : BV(I) \to BV(I)$
has a number of  useful properties.
For simplicity we first assume that $\phi$ is Lipschitz so that we can use more of the specific  methods form \cite{walters}.
In \S 4 we will consider more general results using the subsequent analysis in \cite{hk}, \cite{keller}.

 \begin{lemma}\label{transfer}  Let $\phi : I \to \mathbb R$ be Lipschitz. 
 \begin{enumerate}
 \item 
  There exists a maximal  eigenfunction $h \in BV(I)$ with $h > 0$ such that $\mathcal L_\phi h = e^{P(\phi)} h$ and, in particular, 
$$e^{P(\phi)} = \lim_{n\to +\infty} \left\| \mathcal L_\phi^n1(x) \right\|_\infty^\frac{1}{n}.$$
   \item
 There exists a non-atomic  probability measure $\nu_\phi$  such that $\mathcal L_\phi^*\nu_\phi = e^{P(\phi)} \nu_\phi$.   
 \item There exists $a > 0$ such that 
 $h(x) \geq a > 0$
 \item 
  $h$ is continuous except at the points $\{T^n(\beta) \hbox{ : } n \geq 0\}$.  
 \item 
 There exists $C>0$ and $0 < \rho < 1$ such that  $\|\mathcal L_\phi^n(w) - e^{nP(\phi)} h \nu (w)\| \leq C (\rho e^{P(\phi)})^n \|w\|$ for any $w\in BV(I)$ and $n \geq 1$.
 \end{enumerate}
\end{lemma}

\begin{proof}
The existence of the eigenfunction  $h$ in 
part 1 follows from  (\cite{keller}, Remark 6.8) (as observed in  \cite{baladikeller} on p. 460), where the authors  
also   observe that it follows from their own Theorem 2).  
The existence of the measure in $\nu_\phi$ in    Part 2 follows from Proposition 6.10 of  \cite{keller}
(as also observed in   \cite{baladikeller}, p.460.)
\footnote{See (\cite{walters}, Lemma 3 and Lemma 9) for a proof of  the first 2 parts for a different Banach space and under the additional assumption that $\phi$ is Lipschitz}

Part 3 appears as Part (iii) in Lemma 9 in \cite{walters} 
(which in turn is based on Lemma 1 (i)  in \cite{walters})
 and we briefly recall the proof. 
Let $\mathcal P$ be the partition of $I$ into intervals of the form
$$
\mathcal P 
=
\left\{ 
\left[0, \frac{1}{\beta} \right), \left[\frac{1}{\beta}, \frac{2}{\beta}\right), 
\cdots
 \left[\frac{[1/\beta]-1}{\beta}, \frac{[1/\beta]}{\beta}\right), 
  \left[ \frac{[1/\beta]}{\beta}, 1\right)\right\}.\eqno(2.1)$$
  An interval $J \in \vee_{i=0}^{n-1}T^{-i } \mathcal P$ is called {\it a full interval of rank $n$}  if $T^n: J \to I$ is a bijection. 
Let $N > 0$ then
$I$ is covered by full intervals of rank at least $N$
 (by Lemma 1 (i) in \cite{walters}).
Thus  for any $x$ there  is a sequence $k_j \to +\infty$  of full  intervals $J_j \in  \vee_{i=0}^{k_j-1}T^{-i } \mathcal P$ with
rank at least $k_j$ with 
 $\{x\} = \cap_{j=1}^\infty J_j$.   We claim that there exists 
some $N > 0$ and some interval $J$ 
with $T^N : J \to I$ a  bijection (i.e., $J$ is of full rank $N$) and for which $b:= \inf_{y \in J} h(y) > 0$.
If this was not the case then for all $N$ we  could choose $x$ with $h(x) =0$ and 
since $\mathcal L_\phi^Nh(x) = e^{NP(\phi)} h(x)$
we have that  $h(y) = 0$
 for all preimages $y \in T^{-N}x$, which in turn would imply $h$ is identically zero giving the contradiction.
Thus for $J$ as in this claim, since   by assumption $T^NJ = I$ we have that for $x\in I$:
$$
h(x) = e^{-P(\phi)N} \mathcal L_\phi^Nh(x) \geq e^{-P(\phi)N} b e^{-N \|\phi\|_\infty} = :a.
$$

 Part 4 follows the same lines as the proof of Lemma 9 (iv) in \cite{walters}, and we briefly recall the idea. 
 We observe that $ e^{-nP(\phi)}\mathcal L_{\phi}^n1(x)$ is continuous at points not in $\cup_{i=1}^nT^i(\{\beta\})$
 (where $\{\beta\} = \beta - [\beta]$ denotes the fractional part of $\beta$) and continuous from the right at these points.  The result then follows from $h(x)  = \lim_{n \to +\infty}e^{-nP(\phi)}\mathcal L_{\phi}^n1(x)$.



For Part 5, we first observe that replacing $\phi$ by $\psi = \phi - \log h\circ T + \log h \in BV(I)$
(by virtue of Part 3)
 we have that the associated operator 
satisfies $\mathcal L_\psi  1 =1$ and has spectral radius $1$.
Moreover 
$$
\theta_\psi := \limsup_{n\to +\infty} \left\| \exp\left(\sum_{k=0}^{n-1} \psi(T^k x) \right) \right\|_\infty^\frac{1}{n} < 1.
$$
This corresponds to $\theta_\phi < e^{P(\phi)}$.
By 
Theorem 1 in \cite{baladikeller} we have that
$\mathcal L_{\phi}$ is quasi-compact and 
 the essential spectral radius. 
 is at most  $\theta_\phi$.
Therefore, it suffices to observe  that $e^{P(\phi)}$ 
is a simple eigenvalue and that 
there are no other eigenvalues of absolute value $e^{P(\phi)}$.
 \end{proof}

 It is convenient to consider coboundaries $u\circ T - u$ where $u \in BV(I)$.  
 The following result follows easily from the definitions (and Part (iii) of Lemma 2.2).

\begin{lemma}
We can add constants and coboundaries in $BV(I)$ to $\phi$ without changing the equilibrium state $m_\phi$.  
\footnote{The new function $\phi$ may no longer be Lipschitz since $h$ was not necessarily Lipschitz}
\end{lemma}


In particular,  we can consider $h \in BV(I)$ as in Lemma 2.2 (1) and observe that since 
Lemma 2.2 (3) we have $h \geq a > 0$ we have that $\log h \in BV(I)$.
In particular,  we can replace $\phi$ by $\psi = \phi -  P(\phi)  + \log h - \log h\circ T $  and then we can assume 
 can assume that the associated   transfer operator
$\mathcal L_\psi : BV(I) \to BV(I)$  satisfies $\mathcal L_\psi 1 = 1$.  Thus,  without loss of generality we can assume $P(\phi)=0$.

We now  consider a simple  lemma, which is a special case  of the well known Pinsker inequality.

\begin{lemma}[Pinsker Inequality]
Given probability vectors $\underline q = (q_1, \cdots, q_k)$ and  $\underline p = (p_1, \cdots, p_k)$
we then have the basic inequality

$$
-\sum_{i=1}^k q_i \log q_i + \sum_{i=1}^k q_i \log p_i
\leq 
-\frac{1}{2} \sum_{i=1}^k |p_i - q_i|^2.  \eqno(2.2)
$$
\end{lemma}
\bigskip


We can now follow the lines of the  standard proof of the variational principle
(cf. \cite{walters}).  
Let $\mathcal P = \{P_i\}_{i=1}^k$ be the partition of $I$ into intervals 
given in (2.1) 
   where $k = k(x) =  [1/\beta]$ or $[1/\beta] + 1$, as appropriate.
Since this is clearly a generating partition we have that  the entropy satisfies $h(\mu) = H_\mu (\mathcal P | T^{-1}\mathcal P)$ \cite{walters}.
We can also make the following choices:.
\begin{enumerate}
\item
Given $x \in I$ we can  let 
$p_1(x), \cdots p_{k(x)}(x)$ take the values 
$\{e^{\psi(y)} \hbox{ : }Ty=x\}$; and 
\item
For any  $T$-invariant probability measure $\mu$ we 
 let $q_1(x), \cdots,  q_k(x)$ ($1\leq i \leq k(x)$)
take the values 
 $\mu( \mathcal P | T^{-1} \mathcal B)(y)$ where $T(y) = x$
 for a.e. ($\mu$) $x \in X$ where  $\mathcal B$ is the Borel  sigma algebra. \end{enumerate}
\noindent
We can substitute these choices  into (2.2) and integrate with respect to $\mu$ to get:
$$
\begin{aligned}
&h(T, \mu) + \int \psi(x) d\mu(x)\cr
&=
- \int 
\left(
\sum_{y \in T^{-1}}  \mu(\mathcal P | T^{-1} \mathcal B)(y) \log  \mu(\mathcal P | T^{-1} \mathcal B)(y)
 + \sum_{y \in T^{-1}x}  \mu(\mathcal P | T^{-1} \mathcal B)(y)  \psi(y)
\right) d\mu(x)\cr
&= - \int 
\left(
\sum_{i=1}^k q_i(x) \log q_i(x)
 + \sum_{i=1}^k  q_i(x)  \psi(x)
\right) d\mu(x)\cr
&\leq 
-\frac{1}{2}
\int  \sum_{i=1}^k |p_i(x) - q_i(x)|^2 d\mu(x).\cr
\end{aligned}
\eqno(2.3)
$$
We can get a slightly weaker, but  more useful,   lower bound by  using  the Cauchy-Schwartz inequality to write 
$$
\left( \int \sum_{i=1}^k  |p_i(x) - q_i(x)| d\mu(x)\right)^2
\leq 
\int  \sum_{i=1}^k  |p_i(x) - q_i(x)|^2 d\mu(x).
\eqno(2.4)
$$
Moreover, we can define the usual  norm  on the dual space $BV(I)^*$ of $BV(I)$ by
$$\|\nu\| = \sup\left\{\left|\int g d\nu \right| \hbox{ : } g \in BV(I) \hbox{ with }  \|g \| \leq 1\right\}$$
for $\nu \in BV(I)^*$.
\noindent 
This leads to the following.

\bigskip
\noindent
\begin{lemma}
$\|\mathcal L_\psi^* \mu -\mu\| \leq \int  \sum_{i=1}^k  |p_i(x) - q_i(x)| d\mu(x)$.
\end{lemma}

\medskip
\noindent
\begin{proof}
Given $g \in B_{BV}$ with $\|g\|_\infty \leq 1$ we have 
$$
\begin{aligned}
\left|\int (\mathcal L_\psi g - g) d\mu \right| 
&\leq 
\int \sum_{i=1}^k g(y) \left| p_i(x)  - q_i(x)\right|
 d\mu(x)\cr
&\leq 
\int \sum_i  \left| p_i(x)  - q_i(x)\right|
 d\mu(x)\cr
\end{aligned}
$$
as required.
\end{proof}
\medskip


Finally, we have the following lemma.

\begin{lemma}
There exists $C_1 > 0$ such that 
$\|\mu - m_\phi\| \leq C_1 \|\mathcal L_\psi^{*}\mu - \mu\|$.
\end{lemma}

\begin{proof}
It is at this point that we use  the result from part 5  of Lemma 2.2  that there
exists $0 < \rho < 1$ such that 
$\mathcal L^{*n}\mu = m + U^{*n}\mu$ where $\|U^{*n}\| = O(\rho^n)$.  From this we conclude that 
the  series  $Q = \sum_{n=0}^\infty U^n$ converges.
We can then write
$$
\begin{aligned}
m_\phi &= \lim_{n \to +\infty} \mathcal L_\psi^{*n}\mu = \mu +  \sum_{n=0}^\infty
\mathcal L_\psi^{*n}(\mathcal L_\psi^{*}-I)\mu.
\end{aligned}
$$
Finally, we can write 
$\|\mu - m_\phi\| \leq \|Q\|.\|(\mathcal L_\psi^{*}-I)\mu\|$.
\end{proof}
\bigskip

Combining (2.3), (2.4) and the inequalities in Lemma 2.6 and 2.7  completes the proof of Theorem 1.2.

\section{A generalization of Theorem \ref{main}}

Theorem \ref{main} for $\beta$-transformations is a special case of a more general result for piecewise monotonic transformation 
where there exist
$b_0=0 < b_1 < \cdots < b_N=1$ such that the restriction $T|_{(b_i, b_{i+1})}: (b_i, b_{i+1}) \to I$ is continuous and strictly monotone.

We recall the following property for $T$

\begin{definition}
We will say that $T$ is topologically exact if the any $\epsilon > 0$ there exists $n$ such that for any $x\in I$ we have $T^n(B(x,\epsilon) = I$.
\end{definition}

\begin{definition}
We say that a function $\phi: I \to \mathbb R$ has summable variation if 
$$
\sum_n\hbox{var}_n(\phi) < +\infty
$$
where $
\hbox{var}_n(\phi)  = \sup \{|\phi(x) - \phi(y)| \hbox{ : } x, y \hbox{ are  in same monotonicity interval of $T^n$} \}
$,
for $n \geq 1$.
\end{definition}

The generalization of Theorem \ref{main} takes the following form:

\begin{thm}\label{generaler}
Let $T$ be a piecewise monotonic transformation which is topologically exact.  Let$\phi: I \to \mathbb R$ be a (continuous) 
function of bounded variation such that either 
\begin{enumerate}
\item[a)] $\psi$ has summable variation; or
\item[b)] $\psi$ is H\"older continuous.
\end{enumerate}
 Then  exists a constant $C_0 = C_0(\phi)$ such that for any $T$-invariant probability $\mu$ we have 
$$
\left| \int f d\mu - \int f dm_\phi \right|
\leq C_0 \|f\|
\sqrt{
P(\phi)
- \left(h(\mu) + \int \phi d\mu\right)} \eqno(3.1)
$$
where $\|\cdot\| $ is the norm of bounded variation.
\end{thm}

The definitiion of bounded variation can be generalized as follows 
\begin{definition}
For $p \geq 1$ we can define the   bounded $p$-variation semi-norm of a function $\psi: I \to \mathbb R$ which  takes the form  
$$
\|\psi\|_{BV} = 
\left( \sup\left\{ \sum_{i=0}^n \left| \psi(x_i) - \psi(x_{i+1})\right|^p
\hbox{ : } 0 = x_0 < x_1 \cdots < x_n < x_{n+1} = 1
 \right\}\right)^{1/p}
$$
and let $\|\psi\|_{L^1} = \int_I |\psi(x)| dx$ denote the $L^1$-norm.  
We let $BV_p(I)$ denote the Banach space 
of measurable functions $\psi: I \to \mathbb C$ with norm $\|\psi\| := \|\psi\|_{BV}
 + \|\psi\|_{L^1}  < +\infty$.
 \end{definition}
 
 An even larger space of functions in \cite{keller1} are the following:
 
 \begin{definition}
 Given $\epsilon > 0$ we denote
 $$
 \hbox{\rm osc}_1(x,\phi,\epsilon) = \hbox{\rm esssup}(\phi | B(x, \epsilon)) -  \hbox{\rm essinf}(\phi | B(x, \epsilon))
 $$
 and then we denote
 $
  \hbox{\rm osc}_1(\phi,\epsilon) := \int_I  \hbox{\rm osc}_1(x,\phi,\epsilon) dx
 $. Fix  $\alpha > 0$ and then  for $\epsilon_0 > 0$ we can then write
 $$
 \|\phi\|_{\alpha,1} := \sup_{0 < \epsilon \leq \epsilon_0} \frac{  \hbox{\rm osc}_1(\phi,\epsilon)}{\epsilon^\alpha}.
 $$
 We can then define a norm $\|\phi\| =  \|\phi\|_{\alpha,1} +  \|\phi\|_{L^1}$ and let $H^{\alpha,1}$ be the associated Banach space (see \cite{keller1}, Theorem 1.13,b).
 \end{definition}
 
 The following relationships between these spaces come from \cite{keller1} and \cite{lr}
 \begin{lemma}. Let  $p = \frac{1}{\alpha}$.
 \begin{enumerate}
 \item
 If $\phi: I \to \mathbb R$ is $\alpha$-H\"older continuous then
 $\psi \in BV_p(I)$.
 \item $BV_p \subset H^{\alpha,1}$.
 \end{enumerate}
 \end{lemma}

The proof of Theorem \ref{generaler} requires a suitable generalization of Lemma \ref{transfer}.
 
 \begin{lemma}\label{transferer} Let $T$ be a piecewise monotonic transformation which is topologically exact.  Let $\phi: I \to \mathbb R$ be a (continuous) 
function of bounded variation such that either 
\begin{enumerate}
\item[a)] $\psi$ has $p$-summable variation; or
\item[b)] $\psi$ is $\alpha$-H\"older continuous.
\end{enumerate}
Let $\mathcal L_\phi$ be the associated transfer operator  on $BV_p$ and $H^{\alpha,1}$, respectively.
 \begin{enumerate}
 \item 
  There exists a maximal  eigenfunction $h \in BV(I)$ with $h > 0$ such that $\mathcal L_\phi h = e^{P(\phi)} h$ and, in particular, 
$$e^{P(\phi)} = \lim_{n\to +\infty} \left\| \mathcal L_\phi^n1(x) \right\|_\infty^\frac{1}{n}.$$
   \item
 There exists a non-atomic  probability measure $\nu_\phi$  such that $\mathcal L_\phi^*\nu_\phi = e^{P(\phi)} \nu_\phi$.   
 \item There exists $a > 0$ such that 
 $h(x) \geq a > 0$
 \item 
 There exists $C>0$ and $0 < \rho < 1$ such that  $\|\mathcal L_\phi^n(w) - e^{nP(\phi)} h \nu (w)\| \leq C (\rho e^{P(\phi)})^n \|w\|$ for any $w\in BV(I)$ and $n \geq 1$.
 \end{enumerate}
\end{lemma}

\begin{proof}

Under hypothesis a) the results follow from the results in \cite{hk}.
Under hypothesis b) the results follow from the results in \cite{keller1}.
 We briefly recall the main ideas.  

The measure $\nu_\phi$ in Part 2 occurs as a  
fixed point for the map $\nu \mapsto \mathcal L_\phi\nu/\nu(1)$ on the space of probability measures, i.e., 
$\mathcal L_\phi \nu_\phi = \lambda \nu_\phi$, where $\lambda =\nu_\phi(1)$
(see  p.135 of \cite{hk}).  
Later one can identify $\lambda = e^{P(\phi)}$.

Let  $g(x) := e^{\phi(x)}/\log \lambda$ and then in each of the two cases one show that for large enough $n$
we have 
$$\|\prod_{i=0}^{n-1}g(T^ix)\|_\infty < 1$$ 
(see \cite{hk}, pp. 135-136). This implies that 
the operator  $P = \mathcal L_{\log g}$ satisfies a Lasota-Yorke inequality (see \cite{hk}, Lemma 7), i.e.,
there exists $0 < \rho < 1$ and $\beta > 0$ such that
\begin{enumerate}
\item[a)]
 under hypothesis a) 
$$
\|P^n f\|_{BV} \leq \beta \|f\|_{L^1} + \rho \| f\|_{BV}, \quad n \geq 0; 
$$
\item[b)] 
 under hypothesis b) 
$$
\|P^n f\|_{1,\alpha} \leq \beta \|f\|_{L^1} + \rho \| f\|_{1,\alpha}, \quad n \geq 0.
$$
\end{enumerate}
Moreover, the unit balls in $BV(I)$ and $H^{1,\alpha}$ are compact in the $L^1$-norm 
(by \cite{hk}, Lemma 5 and  \cite{keller1} respectively). 
This leads to the  quasi-compactness of the transfer operator $\mathcal L_{\phi}$ on the respective spaces, i.e., $\mathcal L_{\phi}$ has spectral radius 
$e^{P(\phi)}$  and essential spectral radius at most $0 < \rho < 1$.
The hypothesis of topological exactness implies that $e^{P(\phi)}$ is a simple eigenvalue and there 
are no other eigenvalues of modulus  $e^{P(\phi)}$.  This is nicely explained in the proof of Corollary 4.4 in \cite{lr}.   If $\pi_\phi$ is the one dimensional eigenprojection associated to 
$e^{P(\phi)}$ then we can let $h = \pi_\phi(1)$, the image of the constant function $1$, in Part 1.
Part 4 is a standard  application of the Ionescu-Tulcea and Marinescu theorem \cite{im}.

 \end{proof}
 
 \begin{rem}
 Similar results will hold for transformations $T: I \to I$ with a finite number of monotone branches
 providing there are additional hypotheses which ensures part 3 of the lemma.  
 \end{rem}

\begin{rem}
More generally, it would be sufficient to assume that $g: I \to \mathbb R$ has bounded $p$-variation, which would include the case of the $g$ being H\"older continuous.  
\end{rem}

\section{A Ruhr-Sarig type local result}

In the case that in (1.1) that $P(\phi) - (h(\mu) + \int \phi d\mu)$ is sufficiently small then a sightly different 
bound can be given by modifying the  proof in \cite{ruhrsarig}.



We can consider the pressure
$$
\begin{aligned}
P(\phi) &= \sup\{h(m) + \int \phi dm \hbox{ : } m = \hbox{ $T$-invariant }\}\cr
&=\log \rho(\mathcal L_\phi)
\end{aligned}
$$
(where $\rho(\cdot)$ is the spectral radius)
and for $t \in (-\epsilon, \epsilon)$
\begin{enumerate}
\item
The function $t \mapsto p(t) = P(\phi + t \psi)$ is analytic.
\item 
$\frac{d P(\phi + t \psi)}{dt}|_{t=0} = \int \psi d\mu_\phi =:a_0$
\item 
If  $\psi$ is not cohomologous to a coboundary plus a constant $\frac{d P^2(\phi + t \psi)}{dt^2}|_{t=0} < 0$ and
$P(\phi + t \psi)$ is convex in a neighbourhood of $0$.   
\end{enumerate}
Provided that  $a_1$ is sufficiently close to $a_0$ we can use the above properties to choose $t$
(close to $0$) such that 
$\frac{d P(\phi + t \psi)}{dt}|_{t=0} = \int \psi d\mu_\phi =:a_1$.

We can now  introduce the following restricted pressure function.

\begin{definition}
For $a \in \mathbb R$ we define
$$
Q(a) = \sup \left\{h(\mu) + \int \phi d\mu   \hbox{ : } \int \psi d\mu = a \right\}
$$
which is well defined provided $ \inf_m \{  \int \psi dm \} \leq a \leq \sup_m \{ \int \psi dm\}$
\end{definition}

In particular, we observe $q(a) \leq P(\phi)$.   Since the function $a \mapsto Q(a)$
is analytic we can deduce 
$\frac{d Q(a)}{da}|_{a_0}=0$

We can use the Taylor expansion at $a=a_0$ to write
$$
Q(a_0) - Q(a_1) = Q'(a_0) + Q''(a_0) (a_1-a_0)^2(1+o(1)).
$$

The function $Q$ is actually the Legendre transform of $P$.  More precisely,
$$
\begin{aligned}
P(t) &= h(m_t) + \int (\phi + t\psi) dm_t\cr
&= \underbrace{h(m_t) + \int \phi  dm_t }_{=: Q(t)} + t \int\psi dm_t  \cr
\end{aligned}
$$
where $m_t$ is the equilibrium state of $\phi+t\psi$.  This allows us to deduce that 
$\frac{d Q^2(a)}{dt^2}|_{a=a_0} = \frac{d P^2(t)}{dt^2}|_{t=t_0}$.  

Since 
$$
Q(a_1) \geq h(\nu) + \int \phi d\nu
$$
since $\int \psi d\nu = a_1$ this implies
$$
\begin{aligned}
P(\phi) - \left(h(\nu) + \int \phi d\nu \right) &\geq Q(a_0) - Q(a_1)\cr
&= \frac{d Q^2(a)}{dt^2}|_{a=a_0} (a_1 - a_0)^2 \left(1+o(1)\right)
\end{aligned}
$$

Finally, we conclude that for 
$\int \psi d\nu$ is sufficiently close to 
$\int \psi d\mu_\phi$ then we can bound 

$$
\left| \int f d\mu - \int f dm_\phi \right|
\leq (1+o(1)) 
\sqrt{|\frac{d P^2(t)}{dt^2}|_{t=0}|}
\sqrt{
P(\phi)
- \left(h(\mu) + \int \phi d\mu\right)} \eqno(1.1)
$$

\section{Miscellaneous Comments}

(a) 
The original applications of these pressure results  was to  subshifts of finite type and Axiom A diffeomorphisms \cite{ruhr}.\cite{kadyrov}  
However, by using a simple model by suspension flows \cite{bowen} the corresponding 
result also extends to Axiom A flows.  More precisely, assume that $\phi_t: \Lambda \to \Lambda$ is a $C^1$ Axiom A flow
restricted to a basic set,  $m_\phi$ is a $\phi$-invariant equilibrium state for a H\"older continuous potential $\phi: \Lambda \to \mathbb R$
and $F: \Lambda \to \mathbb R$ is H\"older continuous then 
$$\left| \int F d\mu - \int F dm_\phi \right|
\leq C \|F\|
\sqrt{
P(\phi)
- \left(h(\mu) + \int \phi d\mu\right)} 
$$

\medskip
\noindent
(b) The proof used the strong estimate in Part 5 of Lemma 2.2 to define  $Q$ in the proof of Lemma 2.6.
  However, under any weaker bounds on $\|U^n\| \to 0$  such that the series $Q = \sum_{n=0}^\infty U^n$ converges the same argument will hold.  
  
  \medskip
\noindent
(c)  It may be possible to extend the result to higher dimensional transformations with singularities.
 In light of \cite{ruhrsarig} one might ask if $\|f\|$ can be replaced by the variance $\sigma^2(f)$.

  \medskip
\noindent
(d) 
Ruhr and Sarig have a corresponding result for subshifts  where $\|f\|$ is replaced by an expression involving the variance
$\sigma^2(f)$  which gives a more refined estimate.    
It is a natural question to ask if  this is also true for  (1.1).

\end{document}